\documentclass[11pt, bezier]{article}
\usepackage{amsmath}
\usepackage{amssymb}

\makeatletter

\@ifundefined{date}{}{\date{}}

\newcommand{\lyxmathsym}[1]{\ifmmode\begingroup\def\b@ld{bold}
  \text{\ifx\math@version\b@ld\bfseries\fi#1}\endgroup\else#1\fi}

\usepackage{amsfonts}\usepackage{xspace}
\newtheorem{prethm}{{\bf Theorem}}

\newenvironment{thm}{\begin{prethm}{\hspace{-0.5
               em}{\bf.}}}{\end{prethm}}
\newtheorem{precor}{{\bf Corollary}}

\newenvironment{cor}{\begin{precor}{\hspace{-0.5
               em}{\bf.}}}{\end{precor}}

\newtheorem{prethmm}{{\bf Theorem}}

\newenvironment{thmm}{\begin{prethmm}{\hspace{-1.8
               mm}{\bf.}}}{\end{prethmm}}
               
\newtheorem{preconj}{{\bf Conjecture}}

\newtheorem{preprop}{{\bf Proposition}}

\newenvironment{prop}{\begin{preprop}{\hspace{-0.5
               em}{\bf.}}}{\end{preprop}}

\newtheorem{preques}{{\bf Question}}

\newenvironment{ques}{\begin{preques}{\hspace{-0.5
               em}{\bf.}}}{\end{preques}}

\newtheorem{preremark}{{\bf Remark}}

\newenvironment{remark}{\begin{preremark}\rm{\hspace{-0.5
               em}{\bf.}}}{\end{preremark}}

\newtheorem{prelem}{{\bf Lemma}}

\newenvironment{lem}{\begin{prelem}{\hspace{-0.5
               em}{\bf.}}}{\end{prelem}}

\newtheorem{preproof}{{\bf Proof.}}

\newenvironment{proof}[1]{\begin{preproof}{\rm
               #1}\hfill{$\Box$}}{\end{preproof}}

\usepackage{color}

\begin{document}

\title{On the existence of $d$-homogeneous $3$-way Steiner trades}
\maketitle
\vspace{-1.5cm}
\begin{center}
{\begin{large}
H. Amjadi and N. Soltankhah $^\ast$
\end{large} }\\
\vspace*{0.5cm}
{\it  Faculty of Mathematical Sciences
\\ Alzahra University
\\Vanak Square 19834, Tehran, Iran}\\
\footnote{ Email addresses: {\tt h.amjadi@alzahra.ac.ir} (H. Amjadi), {\tt soltan@alzahra.ac.ir} (N. Soltankhah)}
\footnote{\it $^\ast$ Corresponding author}
\end{center}
%
\begin{abstract}
A $\mu$-way $(v, k, t)$ trade $T = \{T_{1} , T_{2}, . . ., T_{\mu} \}$ of volume $m$ consists of $\mu$ disjoint collections $T_{1}, T_{2}, . . ., T_{\mu}$, each of $m$ blocks of size $k$, such that for every $t$-subset of $v$-set $V$ the number of blocks containing this $t$-subset is the same in each $T_{i}$ $($for $1 \leq i \leq \mu)$. A $\mu$-way $(v, k, t)$ trade is called $\mu$-way $(v, k, t)$ Steiner trade if any $t$-subset of found$(T)$ occurs at most once in $T_{1}$ $(T_{j},~ j \geq 2)$.
A $\mu$-way $(v,k,t)$ trade is called  $d$-homogeneous if each element of $V$ occurs in precisely $d$ blocks of $T_{1}$ $(T_{j},~ j \geq 2)$. In this paper we characterize the $3$-way $3$-homogeneous $(v,3,2)$ Steiner trades of volume $v$. Also we show how to construct a $3$-way $d$-homogeneous $(v,3,2)$ Steiner trade for $d\in \{4,5,6\}$, except for seven small values of $v$.
\end{abstract}
\begin{itemize}
\item[]{{\footnotesize {\bf Key words:}\ Steiner trade, $\mu$-way trade, Homogeneous trade.}}
\item[]{ {\footnotesize {\bf Subject classification: 05B05}.}}
\end{itemize}
\section{Introduction and preliminary results}
Let $V$ be a set of $v$ elements and $k$ and $t$ be two positive integers such that $t < k < v$. A {\sf $(v, k, t)$ trade} $T = \{T_{1} , T_{2} \}$ of {\sf volume $m$} consists of two disjoint collections $T_{1}$ and $T_{2}$, each of containing $m$, $k$-subsets of $V$, called blocks, such that every $t$-subset of $V$ is contained in the same number of blocks in $T_{1}$ and $T_{2}$. A $(v, k, t)$ trade is called {\sf $(v, k, t)$ Steiner trade} if any $t$-subset of $V$ occurs in at most once in $T_{1} (T_{2})$. In a $(v, k, t)$ trade, both collections of blocks must cover the same set of elements. This set of elements is called the {\sf foundation} of the trade and is denoted by {\sf found$(T)$}.

The concept of $\mu$-way $(v, k, t)$ trade, was defined recently in {\cite{Rashidi-3way}}.
\\ \textbf{Definition}: A {\sf $\mu$-way $(v, k, t)$ trade} $T = \{T_{1} , T_{2}, . . ., T_{\mu} \}$ of volume $m$ consists of $\mu$ disjoint collections $T_{1}, T_{2}, . . ., T_{\mu}$, each of $m$ blocks of size $k$, such that for every $t$-subset of $v$-set $V$ the number of blocks containing this $t$-subset is the same in each $T_{i}$ $($for $1 \leq i \leq \mu)$. In other words any pair of collections $\{T_{i}, T_{j}\}$, $1 \leq i < j \leq \mu$ is a $(v, k, t)$ trade of volume $m$. It is clear by the definition that a trade is a $2$-way trade. A $\mu$-way $(v, k, t)$ trade is called {\sf $\mu$-way $(v, k, t)$ Steiner trade} if any $t$-subset of found$(T)$ occurs at most once in $T_{1}$ $(T_{j},~ j \geq 2)$.
A $\mu$-way $(v,k,t)$ trade is called {\sf $d$-homogeneous} if each element of $V$ occurs in precisely $d$ blocks of $T_{1}$ $(T_{j},~ j \geq 2)$. Let $x\in$ found$(T)$, the number of blocks in $T_{i}$ $($for $1\leq i\leq \mu)$ which contains $x$ is denoted by $r_{x}$. The set of blocks in $T_{i}$ $($for $1\leq i\leq \mu)$ which contains $x$ is denoted by $T_{ix}$ $($for $1\leq i\leq \mu)$. It is easy to see that $T_{x}=\{T_{1x} , . . . , T_{\mu x}\}$ is a $\mu$-way $(v, k, t-1)$ trade of volume $r_{x}$. If we remove $x$ from the blocks of $T_{x}$ , then the result will be a $\mu$-way $(v-1, k-1, t-1)$ trade which is called the {\sf derived trade} of $T$. Also it is easy to show that if $T$ is a Steiner trade then its derived trade is also a Steiner trade. A trade $T' = \{T'_{1} , T'_{2}, . . ., T'_{\mu} \}$ is called a {\sf subtrade} of $T = \{T_{1} , T_{2}, . . ., T_{\mu} \}$, if $T'_{i}\subseteq T_{i}$ for $1\leq i\leq \mu$.

For $\mu =2$, Cavenagh et al. {\cite{minimal}}, constructed minimal $d$-homogeneous $(v,3,2)$ Steiner trades of foundation
$v$ and volume $dv/3$ for sufficiently large values of $v$, (specifically, $v > 3(1.75d^{2}+3)$ if $v$ is divisible by $3$ and 
$v > d(4^{d/3+1}+1)$ otherwise).

Generally we can ask the following question.
\begin{ques}\label{ques1}
For given $\mu$, $d$ and $v$, does there exist a $\mu$-way $d$-homogeneous $(v,3,2)$ Steiner trade?
\end{ques}

In this paper, we aim to construct $3$-way $d$-homogeneous $(v,3,2)$ Steiner trades. The Latin trades are a useful tool for building these trades when $v\equiv 0$ (mod $3$), so we use some obtained results on $3$-way $d$-homogeneous Latin trades.

A {\sf Latin square} of order $n$ is an $n\times n$ array $L = ({\ell_{ij}})$ usually on the set $N=\{1,2,...,n\} $ where each element of $N$ appears exactly once in each row and exactly once in each column. We can represent each Latin square as a subset of $N \times N\times N$,
$ L = \{(i, j; k)|$ element $k$ is located in position $(i, j)\} $. 
A {\sf partial Latin square} of order $n$ is an $n\times n$ array $P = ({p_{ij}})$ of elements from the set $N$ where each element of $N$ appears at most once in each row and at most once in each column. The set $S_{P} = \{ (i, j)|~ (i, j; k) \in P \}$  of the partial Latin square $P$ is called the {\sf shape} of $P$ and $|S_{P}|$ is called the {\sf volume} of $P$. 

A {\sf $\mu$-way Latin trade}, $(L_{1}, L_{2}, . . . , L_{\mu})$, of volume $s$ is a collection of $\mu$ partial Latin squares $L_{1}, L_{2}, . . . , L_{\mu}$ containing exactly the same $s$ filled cells, such that if cell $(i, j)$ is filled, it contains a different entry in each of the $\mu$ partial Latin squares, and such that row $i$ in each of the $\mu$ partial Latin squares contains, set-wise, the same symbols and column $j$, likewise. Adams et al.{\cite{3Latin-trade}} studied $\mu$-way Latin trades.
A $\mu$-way Latin trade which is obtained from another one by deleting its empty rows and empty columns, is called a {\sf $\mu$-way $d$-homogeneous Latin trade} $($for $\mu\leq d)$ or briefly a {\sf $(\mu, d, m)$} Latin trade, if it has $m$ rows and in each row and each column $L_{r}$ for $1 \leq r \leq \mu$, contains exactly $d$ elements, and each element appears in $L_{r}$ exactly $d$ times.
Bagheri et al.{\cite{bagh}}, studied the $\mu$-way $d$-homogeneous Latin trades and their main result is as follows:
\begin{thmm}\textsc{{\textsc{\cite{bagh}}}}{\label{Latin trade}}
All $(3, d, m)$ Latin trades $($for $4\leq d\leq m)$ exist, for
\begin{itemize}
\item[{\sf (a)}] $d=4$, except for $m = 6$ and $7$ and possibly for $m = 11$,
\item[{\sf (b)}] $d=5$, except possibly for $m=6$,
\item[{\sf (c)}] $6\leq d\leq 13$,
\item[{\sf (d)}] $d=15$,
\item[{\sf (e)}] $d\geq 4$ and $m\geq d^{2}$,
\item[{\sf (f)}] $m$ a multiple of $5$, except possibly for $m=30$,
\item[{\sf (g)}] $m$ a multiple of $7$, except possibly for $m=42$ and $(3,4,7)$ Latin trade.
\end{itemize}
\end{thmm}

All $3$-way $(v,3,2)$ Steiner trades are characterized in {\cite{Rashidi-3way}}. The authors proved that there is no $3$-way $(v,3,2)$ Steiner trade of volumes $1,2,3,4,5$ and $7$. Also they showed that the $3$-way $(v,3,2)$ Steiner trade of volume $6$ is unique (where the number of occurrences of each element is not the same). So the following proposition is clear.
\begin{prop}{\label{[1,7]}}
The $3$-way $d$-homogeneous $(v,3,2)$ Steiner trade of volume $m$ does not exist for $m\in \{1,2,...,7\}$.
\end{prop}
\begin{remark}{\label{dv/3}}
Since the volume of a $\mu$-way $d$-homogeneous $(v,3,2)$ Steiner trade is $dv/3$, at least one of $d$ or $v$ should be multiple of $3$. 
\end{remark}
\begin{remark}{\label{v>=2d+1}}
In a $\mu$-way $d$-homogeneous $(v,3,2)$ Steiner trade, since every element should belong to $d$ blocks and the other elements of these blocks should be different, so $v\geq 2d+1$.  
\end{remark}

\begin{lem}{\label{v1+v2}}
If there exist two $3$-way $d$-homogeneous $(v_{1},3,2)$ and $(v_{2},3,2)$ Steiner trades of volume $m_{1}$ and $m_{2}$, respectively, then we have a $3$-way $d$-homogeneous $(v_{1}+v_{2},3,2)$ Steiner trade of volume $m_{1}+m_{2}$.
\end{lem}
\begin{proof}
{Let $T = \{T_{1} , T_{2}, T_{3} \}$ be a $3$-way $d$-homogeneous $(v_{1},3,2)$ Steiner trade of volume $m_{1}$ and $T' = \{T'_{1} , T'_{2}, T'_{3} \}$ be a $3$-way $d$-homogeneous $(v_{2},3,2)$ Steiner trade of volume $m_{2}$. It is enough to relabel the elements of found$(T')$, such that found$(T)\cap$ found$(T')=\emptyset$. It is clear that $S = \{T_{1}\cup T'_{1} , T_{2}\cup T'_{2}, T_{3}\cup T'_{3} \}$ is a $3$-way $d$-homogeneous $(v_{1}+v_{2},3,2)$ Steiner trade of volume $m_{1}+m_{2}$.}
\end{proof}

The following lemma which is similar to Lemma $2$ of \cite{minimal}, shows how to construct a $3$-way $d$-homogeneous $(3m,3,2)$ Steiner trade, by using $3$-way $d$-homogeneous Latin trade of order $m$.
\begin{lem}{\label{la-to-tr}}
Let $(L_{1},L_{2},L_{3})$ be a $3$-way $d$-homogeneous Latin trade of order $m$. For each $\alpha\in \{1,2,3\}$, define $T_{\alpha}=\{\{i_1,j_2,k_3\}|~(i,j;k)\in L_{\alpha}\}$. Then $T=\{T_{1},T_{2},T_{3}\}$ is a $3$-way $d$-homogeneous $(3m,3,2)$ Steiner trade.
\end{lem}

The following theorem is a consequence of Theorem \ref{Latin trade} and Lemma \ref{la-to-tr}.
\begin{thm} {\label{(3m,3,2)}}
All $3$-way $d$-homogeneous $(3m,3,2)$ Steiner trades $($for $4\leq d\leq m)$ exist for
\begin{itemize}
\item[{\sf (a)}] $d=4$, except possibly for $m = 6$, $7$ and $11$,
\item[{\sf (b)}] $d=5$, except possibly for $m=6$,
\item[{\sf (c)}] $6\leq d\leq 13$,
\item[{\sf (d)}] $d=15$,
\item[{\sf (e)}] $d\geq 4$ and $m\geq d^{2}$,
\item[{\sf (f)}] $m\equiv 0~($mod $5)$, except possibly for $m=30$,
\item[{\sf (g)}] $m\equiv 0~($mod $7)$, except possibly for $m=7~($where $d=4)$ and $m=42$.
\end{itemize}
\end{thm}

The existence of $3$-way $d$-homogeneous $(v,3,2)$ Steiner trades when $v$ is not multiple of $3$, will be investigated later.

A {\sf Steiner triple system} of order $v$ (briefly $STS(v)$) is a pair $(X,B)$ where $X$ is a $v$-set and $B$ is a collection of $3$-subsets of $X$ (called {\sf triple}) such that every pair of distinct elements of $X$ belongs to exactly one triple of $B$. It is well known that a $STS(v)$ exists if and only if $v\equiv 1, 3$ (mod $6$). A {\sf Kirkman triple system} of order $v$ (briefly $KTS(v)$) is a Steiner triple system of order $v$, $(X,B)$ together with a partition $R$ of the set of triples $B$ into subsets $R_{1},R_{2},...,R_{n}$ called parallel classes such that each $R_{i}~($for $i=1, 2, ..., n)$ is a partition of $X$.
\begin{thm}{\label{3STS}}
There exists a $3$-way $ ((v-1)/2)$-homogeneous $(v,3,2)$ Steiner trade of volume $v(v-1)/6$ for every $v\equiv 1, 3$ $($mod $6)$, $v\neq 7$.
\end{thm}
\begin{proof}
{In {\cite{intersection}}, it is shown that there exist $3$ disjoint $STS(v)$ for every $v\equiv 1, 3$ (mod $6$), $v\neq 7$. It is obvious that the $3$-way trade $T=\{T_{1}, T_{2}, T_{3}\}$, where $T_{1}, T_{2}$ and $T_{3}$ are the disjoint $STS(v)$s, is the desired trade.
}
\end{proof}
\section{$3$-way $3$-homogeneous $(v,3,2)$ Steiner trades}
In this section we answer Question \ref{ques1} when $\mu =d=3$. Note that by Remark \ref{dv/3} there is no restriction for $v$ in this case.
\begin{lem}{\label{(9l,3,2)}}
For every $v=9\ell$, where $\ell\in\left\lbrace 1,2,3,\cdots\right\rbrace $, there exists a $3$-way $3$-homogeneous $(v,3,2)$ Steiner trade of volume $v$.
\end{lem}
\begin{proof}
{According to Corollary 1 of {\cite{bagh}}, there exists a $(3,3,m)$ Latin trade if and only if $3|m$. By Lemma \ref{la-to-tr} we can obtain a $3$-way $3$-homogeneous $(v,3,2)$ Steiner trade of volume $v$ from a $(3,3,m)$ Latin trade, where $v=3m$. 
}
\end{proof}
\begin{lem}{\label{(8l,3,2)}}
For every $v=8\ell$, where $\ell\in\left\lbrace 1,2,3,\cdots\right\rbrace $, there exists a $3$-way $3$-homogeneous $(v,3,2)$ Steiner trade of volume $v$.
\end{lem}
\begin{proof}
{The following trade is a $3$-way $3$-homogeneous $(8,3,2)$ Steiner trade.
\begin{center}
$
\begin{array}{c|c c c c c c c c}
 T_{1}& 123 & 147 & 158 & 248 & 267 & 357 & 368 & 456\\
\hline T_{2}& 124 & 138 & 157 & 237 & 268 & 467 & 458 & 356\\
\hline T_{3}& 127 & 135 & 148 & 246 & 238 & 367 & 457 & 568\\
\end{array}
$
\end{center}
So the proof is obvious by Lemma {\ref{v1+v2}}.
}\end{proof}

The last two lemmas and Lemma \ref{v1+v2} yields the following theorem.
\begin{thm}{\label{9l+8l}}
For every nonzero $v=9\ell+8\ell'$ , where $\ell,\ell'\in\left\lbrace 0,1,2,3,\cdots\right\rbrace $, there exists a $3$-way $3$-homogeneous $(v,3,2)$ Steiner trade of volume $v$.
\end{thm}
The following lemma can be used for characterizing $3$-way $3$-homogeneous $(v,3,2)$ Steiner trades of volume $v$.
\begin{lem}{\label{1-factor}}
There exist only two non-isomorphic $3$-way $(v,2,1)$ Steiner trade of volume $3$.
\end{lem}
\begin{proof}
{Let $T=(T_{1}, T_{2}, T_{3})$ be a $3$-way $(v,2,1)$ Steiner trade of volume $3$. Since $T$ is a Steiner trade, $|$found$(T)| =6$ and let found$(T)=\{1,2,...,6\}$. Without loss of generality we can assume that $T_1=\{12,34,56\}$ and $T_2=\{13,26,45\}$. Now if we consider the blocks of the trade as edges in as $6$ vertex graph and color the edges of each trade with a different color, then the first two trades form an alternating colored $6$-cycle. Up to isomorphism, there are only two different one factors in $K_6$ to complete this graph to a $3$-regular graph. So all trademates put together either form a bipartite or a non-bipartite $3$-regular graph. So $T=(T_{1},T_{2},T_{3})$ up to isomorphism, should be as one of the following trades:
\begin{center}
\begin{tabular}{c c }
$
\begin{array}{c|c|c}
T_{1} & T_{2} & T_{3}\\
\hline 12 & 13 & 14\\
34 & 26 & 25\\
56 & 45 & 36
\end{array}
$
&
$
\begin{array}{c|c|c}
T_{1} & T_{2} & T_{3}\\
\hline 12 & 13 & 15\\
34 & 26 & 24\\
56 & 45 & 36
\end{array}
$
\end{tabular}
\end{center}
}
\end{proof}
\begin{thm}{\label{main,d=3}}
Every $3$-way $3$-homogeneous $(v,3,2)$ Steiner trade of volume $v$ contains a $3$-way $3$-homogeneous $(u,3,2)$ Steiner trade of volume $8$ or $9$, as a subtrade.
\end{thm}
\begin{proof}
{Let  $T = \{T_{1} , T_{2}, T_{3} \}$ be a $3$-way $3$-homogeneous $(v,3,2)$ Steiner trade of volume $v$ with found$(T)=\{1,2,...,v\}$. Let $x\in$ found$(T)$, then $T_{x} = \{T_{1x} , T_{2x}, T_{3x} \}$ is a $(v,3,1)$ trade of volume $r_{x}=3$. Without loss of generality we can assume that found${(T_{x})=\{x,1,2,...,6\}}$. According to Lemma \ref{1-factor} there exist only two cases for $T_{x}$.
\begin{center}
\begin{tabular}{c c}
$
\begin{array}{c|c|c}
T_{1x} & T_{2x} & T_{3x}\\
\hline x12 & x13 & x14\\
x34 & x26 & x25\\
x56 & x45 & x36
\end{array}
$
&
$
\begin{array}{c|c|c}
T_{1x} & T_{2x} & T_{3x}\\
\hline x12 & x13 & x15\\
x34 & x26 & x24\\
x56 & x45 & x36
\end{array}
$
\end{tabular}
\end{center}
Let $T$ contains $T_{x}$ which is as the first form.
\\Since $T=\{T_{1}, T_{2}, T_{3}\}$ is a $3$-homogeneous trade, each of $T_{1}$, $T_{2}$ and $T_{3}$ should contain two other blocks containing element $1$. According to definition of Steiner trades, $T_{1}$ cannot contain block $134$ (since it has block $x34$). So there exist three possible cases:
\begin{list}{}{}
\item[1.] Only one of $T_{2}$ and $T_{3}$ contains block $123$ or $124$:
\\Let $T_{2}$ contains block $123$ (If $T_{3}$ contains block $124$, then the same result will be achieved).
\begin{center}
$
\begin{array}{c|c c c c c}
 T_{1} & x12 & x34 & x56 & 13a & 14b\\
\hline T_{2} & x14 & x25 & x36 & 123 & 1ab\\
\hline T_{3} & x13 & x26 & x45 & 12b & 14a
\end{array}
$
\end{center}
There are four possible cases for elements $a$ and $b$:
\begin{itemize}
\item[1.1.] $a,b\notin \{5,6\}$

The other blocks which contain $4$ should be as follows. The pair $3b$ exists in $T_{3}$ and does not exist in $T_{2}$, it is a contradiction with definition of trade.
\begin{center}
$
\begin{array}{c|c c c c c c c}
 T_{1} & x12 & x34 & x56 & 13a & 14b & 45a\\
\hline T_{2} & x14 & x25 & x36 & 123 & 1ab & 45b & 4a3\\
\hline T_{3} & x13 & x26 & x45 & 12b & 14a & 43b
\end{array}
$
\end{center}

\item[1.2.] $a=6$ and $b=5$

The other blocks of $T_{1}$ and $T_{2}$ which contain $6$ should be as follows. It is a contradiction with definition of trade.
\begin{center}
$
\begin{array}{c|c c c c c c}
 T_{1} & x12 & x34 & x56 & 136 & 145 & 624\\
\hline T_{2} & x14 & x25 & x36 & 123 & 156 & 624\\
\hline T_{3} & x13 & x26 & x45 & 125 & 146
\end{array}
$
\end{center}
\item[1.3.] $a\notin \{5,6\}$ and $b=5$

A $3$-way $3$-homogeneous $(8,3,2)$ is achieved.
\begin{center}
$
\begin{array}{c|c c c c c c c c}
 T_{1} & x12 & x34 & x56 & 13a & 145 & 52a & 236 & 46a\\
\hline T_{2} & x14 & x25 & x36 & 123 & 15a & 546 & 26a & 34a\\
\hline T_{3} & x13 & x26 & x45 & 125 & 14a & 56a & 23a & 364
\end{array}
$
\end{center}
\item[1.4.] $a=6$ and $b\notin \{5,6\}$

Only one of the remaining blocks of $T_{2}$ can contain $6$, but we should have $62$, $64$, $65$ is $T_{2}$ which is impossible.

\begin{center}
$
\begin{array}{c|c c c c c}
 T_{1} & x12 & x34 & x56 & 136 & 14b\\
\hline T_{2} & x14 & x25 & x36 & 123 & 16b\\
\hline T_{3} & x13 & x26 & x45 & 12b & 146
\end{array}
$
\end{center}
\end{itemize}
\item[2.] $T_{2}$ and $T_{3}$ contain block $123$ and $124$, respectively:

The other blocks of $T_{2}$ and $T_{3}$ which contain $1$ should be as follows. It is a contradiction with definition of trade.
\begin{center}
$
\begin{array}{c|c c c c c }
 T_{1} & x12 & x34 & x56 & 13a & 14b\\
\hline T_{2} & x14 & x25 & x36 & 123 & 1ab\\
\hline T_{3} & x13 & x26 & x45 & 124 & 1ab
\end{array}
$
\end{center}
\item[3.] $T_{2}$ and $T_{3}$  do not contain the blocks $123$ and $124$:
\begin{center}
$
\begin{array}{c|c c c c c }
 T_{1} & x12 & x34 & x56 & 13a & 14b\\
\hline T_{2} & x14 & x25 & x36 & 13b & 12a\\
\hline T_{3} & x13 & x26 & x45 & 14a & 12b
\end{array}
$
\end{center}
There are four possible cases for elements $a$ and $b$:
\begin{itemize}
\item[3.1.] $a,b\notin \{5,6\}$

A $3$-way $3$-homogeneous $(9,3,2)$ is achieved.
\begin{center}
$
\begin{array}{c|c c c c c c c c c }
 T_{1} & x12 & x34 & x56 & 13a & 14b &36b & 45a & 62a & 52b\\
\hline T_{2} & x14 & x25 & x36 & 13b & 12a & 34a & 45b & 65a & 62b\\
\hline T_{3} & x13 & x26 & x45 & 14a & 12b & 34b & 3a6 & 65b & 52a
\end{array}
$
\end{center}
\item[3.2.] $a=6$ and $b=5$

The other blocks of $T_{1}$, $T_{2}$ and $T_{3}$ which contain $6$ should be as follows. So the other blocks of $T_{2}$ and $T_{3}$ which contain $4$ should be $423$. It is a contradiction with definition of trade.
\begin{center}
$
\begin{array}{c|c c c c c c c }
 T_{1} & x12 & x34 & x56 & 136 & 145 & 624\\
\hline T_{2} & x14 & x25 & x36 & 135 & 126 & 645 & 423\\
\hline T_{3} & x13 & x26 & x45 & 146 & 125 & 635 & 423
\end{array}
$
\end{center}
\item[3.3.] $a\notin \{5,6\}$ and $b=5$

The other blocks which contain $5$ should be as follows. So the other blocks of $T$ which contain $2$ should be as follows. $36$ appears two times in $T_{2}$. It is a contradiction with definition of Steiner trade.
\begin{center}
$
\begin{array}{c|c c c c c c c }
 T_{1} & x12 & x34 & x56 & 13a & 145 & 523 & 26a\\
\hline T_{2} & x14 & x25 & x36 & 135 & 12a & 546 & 236\\
\hline T_{3} & x13 & x26 & x45 & 14a & 125 & 536 & 23a
\end{array}
$
\end{center}
\item[3.4.] $a=6$ and $b\notin \{5,6\}$

The other blocks which contain $6$ should be as follows.  There is a block in $T_{2}$ which contain $45$. $4$ appears three times in $T_{1}$ but the pair $45$ does not appear in blocks of $T_{1}$. It is a contradiction with definition of trade.
\begin{center}
$
\begin{array}{c|c c c c c c }
 T_{1} & x12 & x34 & x56 & 136 & 14b & 624\\
\hline T_{2} & x14 & x25 & x36 & 13b & 126 & 645\\
\hline T_{3} & x13 & x26 & x45 & 146 & 12b & 635
\end{array}
$
\end{center}
\end{itemize}
\end{list}
For the other case, the same result can be obtained by a similar argument. So it can be deduced that if there exists a $3$-way $3$-homogeneous $(v,3,2)$ Steiner trade of volume $v~($for $v\geq 8$), then it contains a $3$-way $3$-homogeneous $(8,3,2)$ or $(9,3,2)$ Steiner trade of volume $8$ or $9$, respectively.
}\end{proof}
The following corollary is the direct result of Theorem {\ref{main,d=3}}.
\begin{cor}
If there exists a $3$-way $3$-homogeneous $(v,3,2)$ Steiner trade of volume $v$, then it can be represented as a union of disjoint $3$-way $3$-homogeneous $(8,3,2)$ or $(9,3,2)$ Steiner trades of volume $8$ or $9$, respectively.
\end{cor}
Define $[a,b]=\{c\in Z|~a\leq c\leq b\}$.
\begin{thm}{\label{exceptions of d=3}}
The $3$-way $3$-homogeneous $(v,3,2)$ Steiner trade of volume $v$ does not exist for $v\in [1,7]\cup [10,15]\cup [19,23]\cup [28,31]\cup [37,39]\cup \{46,47,55\}$.
\end{thm}
\begin{proof}
{By Proposition \ref{[1,7]}, the $3$-way $3$-homogeneous $(v,3,2)$ Steiner trade of volume $v$ does not exist for $v\in [1,7]$. Let there exist a $3$-way $3$-homogeneous $(v,3,2)$ Steiner trade $T$ of volume $v$ for $v\in [10,15]$. By Theorem \ref{main,d=3}, it should contain a $3$-way $3$-homogeneous $(8,3,2)$ or $(9,3,2)$ Steiner trade $T'$ of volume $8$ or $9$, respectively. Therefore, $T\setminus T'$ is a $3$-way $3$-homogeneous $(u,3,2)$ Steiner trade of volume $u$, where $u\in [1,7]$, which is impossible. By same argument, let there exist a $3$-way $3$-homogeneous $(v,3,2)$ Steiner trade $T$ of volume $v$ for $v\in [19,23]$. By Theorem \ref{main,d=3}, it should contain a $3$-way $3$-homogeneous $(8,3,2)$ or $(9,3,2)$ Steiner trade $T'$ of volume $8$ or $9$, respectively. So $T\setminus T'$ is a $3$-way $3$-homogeneous $(u,3,2)$ Steiner trade of volume $u$, where $u\in [10,15]$, which is a contradiction. The same way, we can prove non-existence of the other mentioned trades.
}
\end{proof}
\begin{thm}{\label{d=3}}
For every $v\geq 8$, there exists a $3$-way $3$-homogeneous $(v,3,2)$ Steiner trade of volume $v$, except for $v\in [10,15]\cup[19,23]\cup[28,31]\cup[37,39]\cup\{46,47,55\}$.
\end{thm}
\begin{proof}
{
According to Theorem {\ref{9l+8l}}, it is enough to represent every $v\geq 8$ in the form $9\ell+8\ell'$, where $\ell, \ell'\geq 0$ as follows:
\\$v=9k$, where $k\geq1$ 
\\$v=9k+1=9(k-7)+64$, where $k-7\geq0$
\\$v=9k+2=9(k-6)+56$, where $k-6\geq0$
\\$v=9k+3=9(k-5)+48$, where $k-5\geq0$
\\$v=9k+4=9(k-4)+40$, where $k-4\geq0$
\\$v=9k+5=9(k-3)+32$, where $k-3\geq0$
\\$v=9k+6=9(k-2)+24$, where $k-2\geq0$
\\$v=9k+7=9(k-1)+16$, where $k-1\geq0$
\\$v=9k+8$, where $k\geq0$
\\Using Theorem \ref{exceptions of d=3} completes the proof.
}
\end{proof}
\section{$3$-way $d$-homogeneous $(v,3,2)$ Steiner trade for $d\in \{4,5,6\}$}
In this section we answer Question \ref{ques1} when $\mu =3$ and $d=4,5$ and $6$.
\subsection{$d=4$}
In this subsection we completely answer Question \ref{ques1} for $d=4$. Note that since $d=4$ by Remark \ref{dv/3}, $v$ should be a multiple of $3$.
\begin{prop}{\label{9,18,21,33}}
There exist $3$-way $4$-homogeneous $(9,3,2)$, $(18,3,2)$, $(21,3,2)$ and $(33,3,2)$ Steiner trades of volume $12$, $24$, $28$ and $44$, respectively.
\end{prop}
\begin{proof}
{According to Theorem \ref{3STS}, there exists a $3$-way $4$-homogeneous $(9,3,2)$ Steiner trade of volume $12$.
By Theorem {\ref{(3m,3,2)}}, there exist $3$-way $4$-homogeneous $(12,3,2)$ and $(15,3,2)$ Steiner trades of volume $16$ and $20$, respectively. Since $18=9+9$, $21=12+9$ and $33=15+18$, regarding to Lemma {\ref{v1+v2}}, the result follows.
}
\end{proof}
\begin{thm}{\label{d=4}}
There exists a $3$-way $4$-homogeneous $(v,3,2)$ Steiner trade if and only if $v \geq 9$ and $v\equiv 0~($mod $3)$. 
\end{thm}
\begin{proof}
{The result is followed by Theorem {\ref{(3m,3,2)}}, Proposition {\ref{9,18,21,33}} and Remarks {\ref{dv/3}} and {\ref{v>=2d+1}}.
}
\end{proof}
\subsection{$d=5$}
In this subsection we solve Question \ref{ques1} for $d=5$, except for $v=18$. By Remark \ref{dv/3}, $v$ should be a multiple of $3$.
\begin{prop}{\label{d=5,v=12}}
There exists a $3$-way $5$-homogeneous $(12,3,2)$ Steiner trade of volume $20$.
\end{prop}
\begin{proof}
{It is enough to use three disjoint decomposition of $K_{12}-I$ (the graph obtained from $K_{12}$ by removing the edges of a perfect matching) into triples. In other words, in the literature of block designs, it is enough to consider three disjoint compatible $(2,3)$-packings on $12$ points (see \cite{packing}).
}
\end{proof}
\begin{thm}{\label{d=5}}
Except possibly for $v=18$, there exists a $3$-way $5$-homogeneous $(v,3,2)$ Steiner trade if and only if $v \geq 12$ and $v\equiv 0~($mod $3)$. 
\end{thm}
\begin{proof}
{The result follows by Theorem {\ref{(3m,3,2)}}, Proposition {\ref{d=5,v=12}} and Remarks {\ref{dv/3}} and {\ref{v>=2d+1}}.
}
\end{proof}
\subsection{$d=6$}
In this subsection we solve Question \ref{ques1} for $d=6$, only six values of $v$ are left undecided. Note that by Remark \ref{dv/3} there is no restriction for $v$ in this case.
\begin{prop}{\label{d=6-13,14}}
There exist a $3$-way $6$-homogeneous $(v,3,2)$ Steiner trade of volume $2v$, where $v=13,14,15$ and $16$.
\end{prop}
\begin{proof}
{By Theorem \ref{3STS} there exist a $3$-way $6$-homogeneous $(13,3,2)$ Steiner trade of volume $26$.

Let $X=\{1,2,3,4,5,6,7,8,9,a,b,c,d,e,f\}$ and let $(X,B)$ be a $KTS(15)$, where
\\$B=\{123,48c,5ae,6bd,79f,145,28a,3df,69e,7bc,167,29b,3ce,4af,58d,189,2cf,356,4be,$
\\$ 7ad,1ab,2de,347,59c,68f,1cd,246,39a,5bf,78e,1ef,257,38b,49d,6ac\}$
\\We define the following permutations on $X$:\\
$\pi_{1}=(1~e)(2~4)(3~7)(5~6)(8~b)(a~c)(9~d)$\\
$\pi_{2}=(9~7)(3~d)(4~a)(2~c)(6~8)(5~b)(1~e)$
\\The intersection of three sets $B, \pi_{1} (B)$ and $\pi_{2} (B)$ is $C=\{79f,3df,4af,2cf,68f,5bf,1ef\}$.
\\So $T=\{T_{1}, T_{2}, T_{3}\}$ is a $3$-way $6$-homogeneous $(14,3,2)$ Steiner trade of volume $28$, where
$T_{1}=B\setminus C$, $T_{2}=\pi_{1} (B)\setminus C$ and $T_{3}=\pi_{3} (B)\setminus C$.

Let $X=\{1,2,3,4,5,6,7,8,9,a,b,c,d,e,f\}$ and let $(X,B_{1}\cup B_{2})$ be a $KTS(15)$, where $B_{1}=\{12f,345,678,9ab,cde\}$ and
$B_{2}=\{36f,15e,24a,7bc,89d,9cf,147,2be,38a,56d,7af,$
\\$16c,258,3bd,49e,4df,18b,269,37e,5ac,8ef,1ad,23c,46b,579,5bf,139,27d,48c,6ae\}$
\\We define permutations $\pi=(6~7~8)(9~b~a)(c~d~e)$ on $X$:
\\$B_{2}$, $\pi (B_{2})$ and $\pi (\pi (B_{2}))$ are disjoint. So $T=\{B_{2},\pi (B_{2}),\pi (\pi (B_{2}))\}$ is a 
$3$-way $6$-homogeneous $(15,3,2)$ Steiner trade of volume $30$.

Let $Y=\{1,2,3,4,5,6,7,8,9,a,b,c,d,e,f,g\}$ and
\\$D=\{59e,6cg,7af,8bd,5bf,6ad,7ce,89g,19f,2bg,3cd,4ae,1ag,2cf,3be,49d,15d,28e,36f,$
\\$47g,16e,27d,35g,48f,17b,269,38a,45c,18c,25a,379,46b\}$
\\(In fact $D$ is the block set of a Kirkman frame of type $4^{4}$ with group set 
\\$G=\{1234,5678,9abc,defg\}$ \cite{survey-Stinson}). We define the following permutations on $Y$:
\\$\pi_{3}=(1~2)(3~4)(5~6)(7~8)$
\\$\pi_{4}=(1~3)(2~4)(5~7)(6~8)(9~a)(b~c)$
\\$D$, $\pi_{3}(D)$ and $\pi_{4}(D)$ are disjoint. So $T=\{D,\pi_{3}(D),\pi_{4}(D)\}$ is a 
$3$-way $6$-homogeneous $(16,3,2)$ Steiner trade of volume $32$.
}\end{proof}
\begin{thm}{\label{d=6}}
There exists a $3$-way $6$-homogeneous $(v,3,2)$ Steiner trade of volume $2v$ for every $v\geq 13$, except possibly for $v\in \{17,19,20,22,23,25\}$.
\begin{proof}
{By Remark {\ref{v>=2d+1}}, $v\geq 13$.
We investigate three cases for $v$. For $v=3\ell$, where $\ell\geq 6$, the result clearly follows by Theorem {\ref{(3m,3,2)}}. For cases $v=3\ell +1=3(\ell-4)+13$ and $v=3\ell +2=3(\ell-4)+14$, where $\ell\geq 10$, we use Theorem {\ref{(3m,3,2)}} and Proposition {\ref{d=6-13,14}}.
\\For $v\in \{13,14,15,16,26,28,29\}$, by Proposition {\ref{d=6-13,14}} and Lemma {\ref{v1+v2}}, there exists a $3$-way $6$-homogeneous $(v,3,2)$ Steiner trade of volume $2v$.
}
\end{proof}
\end{thm}

Our results are summarized below:

\textbf{Main Theorem.}
All $3$-way $d$-homogeneous $(v,3,2)$ Steiner trades exist for
\begin{itemize}
\item[\textsc{(I)}] $v=3m$:
\begin{itemize}
\item[{\sf (a)}] $d=4$, $m\geq 3$, (by Theorem \ref{d=4})
\item[{\sf (b)}] $d=5$, $m\geq 4$ except possibly for $m=6$, (by Theorem \ref{d=5})
\item[{\sf (c)}] $7\leq d\leq 13$, $m\geq d$, (by Theorem \ref{(3m,3,2)})
\item[{\sf (d)}] $d=15$, $m\geq d$, (by Theorem \ref{(3m,3,2)})
\item[{\sf (e)}] $d\geq 4$ and $m\geq d^{2}$, (by Theorem \ref{(3m,3,2)})
\item[{\sf (f)}] $m\equiv 0~($mod $5)$ and $m\geq d$, except possibly for $m=30$, (by Theorem \ref{(3m,3,2)})
\item[{\sf (g)}] $m\equiv 0~($mod $7)$ and $m\geq d$, except possibly for $m=7$ (where $d=4$) and $m=42$. (by Theorem \ref{(3m,3,2)})
\end{itemize}
\item[\textsc{(II)}] $v=6m+1$:
\begin{itemize}
\item[] $d=3m$, $m\geq 2$. (by Theorem \ref{3STS})
\end{itemize}
\item[\textsc{(III)}] $v=6m+3$:
\begin{itemize}
\item[] $d=3m+1$, $m\geq 1$. (by Theorem \ref{3STS})
\end{itemize} 
\item[\textsc{(IV)}] All $v$:
\begin{itemize}
\item[\sf (a)] $d=3$, $v\geq 8$ except for $v\in [10,15]\cup[19,23]\cup[28,31]\cup[37,39]\cup\{46,47,55\}$, (by Theorem \ref{d=3})
\item[\sf (b)] $d=6$, $v\geq 13$ except possibly for $v\in \{17,19,20,22,23,25\}$. (by Theorem \ref{d=6})
\end{itemize}
\end{itemize}

\end{document}